\documentclass[12pt, reqno,draft]{amsart}
\usepackage{amsmath}
\usepackage{amssymb}
\usepackage{amsthm}
\usepackage{enumerate}
\usepackage{xcolor}
\usepackage{mathrsfs}
\usepackage[normalem]{ulem}
\usepackage{mathtools}
\usepackage[abbrev]{amsrefs}
\usepackage{graphicx}
\usepackage[all]{xy}
\usepackage[utf8]{inputenc}
\usepackage{ifthen}
\usepackage{tikz}
\usepackage{enumitem}

\newcommand{\Env}[2][]{%
\ifthenelse{ \equal{#1}{} }
{\ensuremath{#2_{\mathsf{c}}}}
{\ensuremath{#2_{\mathsf{c},#1}}}
}

\newcommand{\A}{\ensuremath{\mathcal{A}}}
\newcommand{\RR}{\ensuremath{\mathcal{R}}}
\newcommand{\QQ}{\ensuremath{\mathcal{Q}}}
\newcommand{\VV}{\ensuremath{\mathcal{V}}}
\newcommand{\BB}{\ensuremath{\mathcal{B}}}
\newcommand{\LL}{\ensuremath{\mathscr{L}}}
\newcommand{\C}{\ensuremath{\mathcal{C}}}
\newcommand{\Dy}{\ensuremath{\mathcal{D}}}
\newcommand{\Lip}{\operatorname{\rm{Lip}}}
\newcommand{\Rea}{\ensuremath{\mathbb{R}}}
\newcommand{\Nat}{\ensuremath{\mathbb{N}}}
\newcommand{\Int}{\ensuremath{\mathbb{Z}}}
\newcommand{\MM}{\ensuremath{\mathcal{M}}}
\newcommand{\NN}{\ensuremath{\mathcal{N}}}
\newcommand{\F}{\ensuremath{\mathcal{F}}}
\newcommand{\B}{\ensuremath{\mathcal{B}}}
\newcommand{\Id}{\ensuremath{\mathrm{Id}}}

\newcommand{\co}{\operatorname{\mathrm{co}}\nolimits}
\newcommand{\Ker}{\operatorname{\rm{Ker}}}
\newcommand{\spn}{\operatorname{\rm{span}}}
\newcommand{\sgn}{\operatorname{\rm{sgn}}}

\allowdisplaybreaks

\newtheorem{Theorem}{Theorem}[section]
\newtheorem{Lemma}[Theorem]{Lemma}

\newtheorem{Corollary}[Theorem]{Corollary}

\theoremstyle{remark}

\newtheorem{Definition}[Theorem]{Definition}

\newtheorem{Question}[Theorem]{Question}

\AtBeginDocument{% THIS CAUSES THAT MR NUMBERS ARE OMITTED IN THE REFERENCES
\def\MR#1{}
}



\begin{document}
\title[Structure of the Lipschitz free $p$-spaces $\F_p(\Int^d)$ and $\F_p(\Rea^d)$]
{Structure of the Lipschitz free $p$-spaces $\F_p(\Int^d)$ and $\F_p(\Rea^d)$ for $0<p\le 1$}

\author[F. Albiac]{Fernando Albiac}
\address{Department of Mathematics, Statistics, and Computer Sciences--InaMat$^2$ \\
Universidad P\'ublica de Navarra\\
Campus de Arrosad\'{i}a\\
Pamplona\\
31006 Spain}
\email{fernando.albiac@unavarra.es}

\author[J. L. Ansorena]{Jos\'e L. Ansorena}
\address{Department of Mathematics and Computer Sciences\\
Universidad de La Rioja\\
Logro\~no\\
26004 Spain}
\email{joseluis.ansorena@unirioja.es}

\author[M. C\'uth]{Marek C\'uth}
\address{Faculty of Mathematics and Physics, Department of Mathematical Analysis\\
Charles University\\
186 75 Praha 8\\
Czech Republic}
\email{cuth@karlin.mff.cuni.cz}

\author[M. Doucha]{Michal Doucha}
\address{Institute of Mathematics\\
Czech Academy of Sciences\\
\v Zitn\'a 25\\
115 67 Praha 1\\
Czech Republic}
\email{doucha@math.cas.cz}

\begin{abstract} Our aim in this article is to contribute to the theory of Lipschitz free $p$-spaces for $0<p\le 1$ over the Euclidean spaces $\Rea^d$ and $\Int^d$. To that end, on one hand we show that $\F_p(\Rea^d)$ admits a Schauder basis for every $p\in(0,1]$, thus generalizing the corresponding result for the case $p=1$ by H\'ajek and Perneck\'a \cite{HP2014}*{Theorem 3.1} and answering in the positive a question that was raised in \cite{AACD2019}. Explicit formulas for the bases of both $\F_p(\Rea^d)$ and its isomorphic space $\F_p([0,1]^d)$ are given. On the other hand we show that the well-known fact that $\F(\Int)$ is isomorphic to $\ell_{1}$ does not extend to the case when $p<1$, that is, $\F_{p}(\Int)$ is not isomorphic to $\ell_{p}$ when $0<p<1$.
\end{abstract}

\thanks{F. Albiac acknowledges the support of the Spanish Ministry for Economy and Competitivity under Grant MTM2016-76808-P for \emph{Operators, lattices, and structure of Banach spaces}. F. Albiac and J.~L. Ansorena acknowledge the support of the Spanish Ministry for Science, Innovation, and Universities under Grant PGC2018-095366-B-I00 for \emph{An\'alisis Vectorial, Multilineal y Aproximaci\'on}. M.~C\'uth has been supported by Charles University Research program No. UNCE/SCI/023. M. Doucha was supported by the GA\v{C}R project EXPRO 20-31529X and RVO: 67985840.}

\maketitle

\section{Introduction and background}
\noindent Suppose $0<p\le 1$. Given a pointed $p$-metric space $\MM$ it is possible to construct a unique $p$-Banach space $\F_p(\MM)$ in such a way that $\MM$ embeds isometrically in $\F_{p}(\MM)$ via a canonical map denoted $\delta_\MM$, and for every $p$-Banach space $X$ and every Lipschitz map $f\colon \MM\to X$ with Lipchitz constant $\Lip(f)$ that maps the base point $0$ in $\MM$ to $0\in X$ extends to a unique linear bounded map $T_f\colon \F_p(\MM)\to X$ with $\Vert T_f\Vert =\Lip(f)$. The space $\F_p(\MM)$ is known as the \emph{Lipschitz free $p$-space} over $\MM$. 
This class of $p$-Banach spaces provides a canonical linearization process of Lipschitz maps between $p$-metric spaces: any Lipschitz map $f$ from a $p$-metric space $\MM_1$ to a $p$-metric space $\MM_{2}$ which maps the base point in $\MM_1$ to the base point in $\MM_2$ extends to a continuous linear map $L_f\colon \F_p(\MM_1)\to \F_p(\MM_2)$ with $\Vert L_f\Vert \le \Lip(f)$. 

Lipschitz free $p$-spaces were introduced in \cite{AlbiacKalton2009}, where they were used to provide for every for $0<p<1$ a couple of \emph{separable} $p$-Banach spaces which are Lipschitz-isomorphic without being linearly isomorphic. These spaces constitute a new family of $p$-Banach spaces which are easy to define but whose geometry is difficult to grasp. This task was undertaken by the authors in \cite{AACD2018} and continued in the articles \cites{AACD2019,AACD2020}.

Within this subject, it is specially interesting and challenging to understand the structure of Lipschitz free $p$-spaces over subsets of the metric space $\Rea$ (or more generally over $\Rea^d$ for $d\in\Nat$) endowed with the Euclidean distance (see \cite{AACD2018}*{Comments at the end of section \S5}). Although the papers \cites{AACD2019,AACD2020} do not focus on this kind of Lipschitz free $p$-spaces, they contain results that apply in particular to them. Let us gather the most significant advances concerning the geometry of Lipschitz free $p$-spaces over Euclidean spaces achieved in those papers. Some of them are explicitly stated in the articles \cites{AACD2019,AACD2020} (in which case we provide the reference), while others are straightforward consequences of more general results. In the list below we assume $0<p\le 1$.

\begin{enumerate}[label={\textbf{(A.\arabic*)}}, leftmargin=*,widest=iii]
\item\label{two:net} For every $d\in\Nat$ and every net $\NN$ in $\Rea^d$, $\F_p(\NN)\simeq\F_p(\Int^d)$ (\cite{AACD2019}*{Proposition~3.6}).

\item\label{two:basis} The space $\F_p([0,1])$ has a Schauder basis (\cite{AACD2019}*{Theorem~5.7}).

\item For every $d\in\Nat$, the space $\F_p(\Int^d)$ has a Schauder basis (\cite{AACD2019}*{Theorem~5.3}).

\item\label{three:isomorphismsRd} $\F_p(\Rea^d)\simeq \F_p([0,1])^d\simeq \F_p(\Rea_+^d)\simeq \F_p(S^d)\simeq \ell_p(\F_p(\Rea^d))$ for every $d\in\Nat$ (\cite{AACD2020}*{Theorem~4.15, Corollary~4.17 and Theorem~4.21}).

\item\label{three:isomorphismsNd} $\F_p(\Int^d)\simeq \F_p(\Nat^d) \simeq \ell_p(\F_p(\Int^d))$ for every $d\in\Nat$ (\cite{AACD2020}*{Theorems 5.8 and 5.12}).

\item\label{three:local} The spaces $\F_p(\Int^d)$ and $\F_p(\Rea^d)$, despite being non-isomorphic, have the same local structure (\cite{AACD2020}*{Corollary 5.14}).

\item\label{three:complemeted} For every $d\in\Nat$ there is a constant $C=C(p,d)$ such that for every $\MM\subset\NN\subset\Rea^d$, the space $\F_p(\MM)$ is $C$-complemented in $\F_p(\NN)$ (\cite{AACD2020}*{Corollary 5.3}).

\item\label{three:piproperty} For every $0<p\le 1$, every $d\in\Nat$ and every $\MM\subset\Rea^d$, $\F_p(\MM)$ has the $\pi$-property (\cite{AACD2020}*{Corollary 5.3}).

\item\label{three:lp} For every $\MM\subset\Rea^d$ infinite, there is $\NN\subset\MM$ such that $\F_p(\NN) \simeq \ell_p$ (\cite{AACD2020}*{Theorem 3.2}).

\end{enumerate}

Combining \ref{three:complemeted} with \ref{three:isomorphismsRd}, and using Pe{\l}czy\'nski's decomposition method (see, e.g., \cite{AlbiacKalton2016}*{Theorem 2.2.3}), yields that $\F_p(\NN)\simeq \F_p(\Rea^d)$ whenever the subset $\NN$ of $\Rea^d$ has non-empty interior. So, the research on Lipschitz free $p$-spaces over subsets of Euclidean spaces reduces to the following two main topics (with non-empty intersection).

\begin{enumerate}[label={\textbf{(Q.\alph*)}}, leftmargin=*]
\item\label{QA} The study of the geometry of the spaces $\F_p(\Rea^d)$.

\item\label{QB} To contrast the spaces $\F_p(\NN)$ for different subsets $\NN\subset\Rea^d$ with empty interior.
\end{enumerate}

As far as the topic \ref{QA} is concerned, the main question suggested by the previous work on the subject is whether \ref{two:basis} extends to higher dimensions. Within the scale of approximation properties, the existence of a Schauder basis is the most demanding one. Thus, this question also connects with \ref{three:piproperty}, which states in particular that $\F_p([0,1]^d)$ has the $\pi$-property.

The approximation properties of Lipschitz free spaces (i.e., Lipschitz free $p$-spaces for $p=1$) has attracted a lot of attention from the specialists from the upsurge of interest in the theory back 2003. We refer the reader to \cites{GodefroyKalton2003,Dalet2015,GodefroyOzawa2014,LP2013,AmbrosioPuglisi2016,PS2015,Dalet2015b,Godefroy2015,Godefroy2015b,WojFonf2008} for a non exhaustive list of papers containing positive answers to this question within different frameworks. In contrast, determining whether a given Lipschitz free space has a Schauder basis has shown to be a somewhat elusive task. To the best of our knowledge, the list of papers that achieve positive answers in this direction reduces to \cites{HP2014,HN2017,CuthDoucha2016,DoKa2020}.

As far as Lipschitz-free spaces over Euclidean spaces is concerned, it is known that $\F(\Rea^d)$ has a Schauder basis (see \cite{HP2014}*{Theorem 3.1}). One of the goals of this paper is to see whether analogous statements hold for the more general case of Lipschitz-free $p$-spaces for $p\in(0,1]$. In this spirit, here we extend this result by proving that $\F_p(\Rea^d)$ admits a Schauder basis for every $p\in(0,1]$, thus answering in the positive \cite{AACD2019}*{Question 6.5}. Moreover, exact formulas for the basis of both $\F_p(\Rea^d)$ and its isomorphic space $\F_p([0,1]^d)$ are provided. Section~\ref{Sect:Basis} will be devoted to take care of this (see Theorem~\ref{thm:basis01d} and Theorem~\ref{thm:basisRd}).

As for the topic \ref{QB}, it is natural to initiate the study with uniformly separated subsets. Note that if $\MM\subset \Rea^d$ is uniformly separated, then it is contained in a net $\NN$. Consequently, by \ref{two:net}, \ref{three:lp} and \ref{three:complemeted}, $\ell_p$ is complemented in $\F_p(\MM)$ and $\F_p(\MM)$ is complemented in $\F_p(\Int^d)$. Hence, if $\F_p(\Int^d)$ were isomorphic to $\ell_p$, applying Pe{\l}czy\'nski's decomposition method would give $\ell_p \simeq \F_p(\MM)$. So, our first task should be to determine whether $\F_p(\Int^d)$ is isomorphic to $\ell_p$ or not. It is known (\cite{NaorSchechtman2007}) that, for $d\ge 2$, $\F(\Int^d)$ is not isomorphic to $\ell_1$. Taking envelopes yields $\F_p(\Int^d)\not\simeq\ell_p$ for any $0<p\le 1$ and $d\ge 2$ (see \cite{AACD2018}*{Corollary 4.2}). On the other hand, it is known and easy to prove, that $\F(\Int)\simeq\ell_1$. More generally, we have $\F(\MM)\simeq\ell_1$ whenever $\MM$ is the closure of a zero-measure subset of $\Rea$ (see \cite{CuthDoucha2016}). This result was extended to $0<p<1$ by the authors replacing the Euclidean distance $|\cdot|$ on $\mathbb R$ with its anti-snowflaking $|\cdot|^{1/p}$. That is, we have $\F_p(\Int,|\cdot|^{1/p})\simeq\ell_p$ for every $0<p\le 1$. A question that implicitly arose from \cite{AACD2018} is whether the same holds for the Euclidean distance, i.e., whether $\F_p(\Int)$ is isomorphic to $\ell_p$ or not for $0<p<1$. Section~\ref{sect:PpNotlp} is devoted to providing a negative answer to this problem.  Note that  $\F_p(\Int,|\cdot|^{\alpha})\simeq\ell_p$ for $0<\alpha<1$ and $0<p\le 1$ (see \cite{AACD2020}*{comments preceding Question 8}). So, this new result exhibits a somewhat surprising discontinuity in the behavior of the Lipschitz free $p$-spaces over the family of $p$-metric spaces $(\Int,|\cdot|^\alpha)$
for $0<\alpha\le 1/p$.

\subsection{Terminology}
Throughout this article we use standard facts and notation from quasi-Banach spaces and Lipschitz free $p$-spaces over quasimetric spaces as can be found in \cite{AACD2018}. Nonetheless, we will record the notation that is most heavily used. A \emph{quasi-norm} on a vector space $X$ over the real field $\Rea$ is a map $\Vert \cdot\Vert\colon X\to[0,\infty)$ satisfying $\Vert x\Vert>0$ unless $x=0$, $\Vert t\, x\Vert=|t| \, \Vert x\Vert$ for all $t\in\Rea$ and all $x\in X$, and
\begin{equation}\label{eq:qn}
\Vert x+y\Vert\le \kappa (\Vert x\Vert +\Vert y\Vert), \quad x,\, y\in X.
\end{equation}
for some constant $\kappa\ge 1$.
The optimal constant such that \eqref{eq:qn} will be called the \emph{modulus of concavity} of $X$.

Let $0<p\le 1$. If $\Vert \cdot\Vert$ fulfils the condition
\[
\Vert x+y\Vert^p \le \Vert x\Vert^p+\Vert y\Vert^p, \quad x,\, y\in X,
\]
we say that $\Vert \cdot\Vert$ is a \emph{$p$-norm.} Any $p$-norm is a quasi-norm with modulus of concavity at most $2^{1/p-1}$. A quasi-norm induces a Hausdorff vector topology on $X$. If $X$ is a complete topological vector space, we say that $(X,\Vert \cdot\Vert)$ is a \emph{quasi-Banach space}. A $p$-Banach space will be a quasi-Banach space equipped with a $p$-norm.

A quasi-Banach space $X$ is said to have the \emph{bounded approximation property} (BAP for short) if there exists a net $(T_\alpha)_{\alpha\in\A}$ consisting of finite-rank linear operators with
\begin{equation*}
\sup_{\alpha\in\A} \Vert T_\alpha\Vert <\infty
\end{equation*}
that converges to $\Id_X$ uniformly on compact sets. If, moreover, each operator $T_\alpha$ is a projection we say that $X$ has the \emph{$\pi$-property}.

A \emph{Schauder basis} of a quasi-Banach space $X$ is a sequence $(x_n)_{n=1}^\infty$ in $X$ such that for every $x\in X$ there is a unique sequence $(a_n)_{n=1}^\infty$ with $x=\sum_{n=1}^\infty a_n\, x_n$. Associated to the Schauder basis $(x_n)_{n=1}^\infty$ the partial-sum projections $P_m\colon X\to X$, $m\in\Nat$, given by
\[
x=\sum_{n=1}^\infty a_n\, x_n\mapsto P_m(x)= \sum_{n=1}^m a_n\, x_n,
\]
are uniformly bounded. Therefore, if a quasi-Banach space $X$ has a Schauder basis, then it has both the BAP and the $\pi$-property. Conversely, given a sequence $(P_m)_{m=1}^\infty$ of linear maps from $X$ into $X$ such that $\sup_m \Vert P_m\Vert<\infty$, $\cup_{m=1}^\infty P_m(X)$ is dense in $X$, $\dim(P_m(X))=m$, and $P_m\circ P_n =P_{\min\{n,m\}}$ for all $n$, $m\in\Nat$, there is a Schauder basis whose associated projections are $(P_m)_{m=1}^\infty$. Namely, if for each $n\in\Nat$ we pick an arbitrary non-zero vector $x_n$ in the one-dimensional space $P_n(X)\cap \Ker(P_{n-1})$, then $(x_n)_{n=1}^\infty$ is such a Schauder basis.

We say that a quasi-Banach space $X$ is \emph{$K$-complemented} in $Y$ if there are bounded linear maps $S\colon X\to Y$ and $P\colon Y\to X$ with $P\circ S=\Id_X$ and $\Vert S\Vert\, \Vert T\Vert\le K$. If, moreover, $S\circ P=\Id_Y$, the spaces $X$ and $Y$ are said to be \emph{$K$-isomorphic.} In the case when $S$ is the inclusion map, that is, $P$ is a projection onto $X$, we say that $X$ is a $K$-complemented subspace of $Y$. If the constant $K$ is irrelevant, we simply drop it from the notation. We can define similar notions replacing the quasi-Banach spaces $X$ and $Y$ with metric, or quasi-metric, spaces $\MM$ and $\NN$, and replacing bounded linear maps with Lipschitz maps. In the ``metric case'', we will say that $\MM$ is a Lipschitz retract of $\NN$ with constant $K$, or that $\MM$ and $\NN$ are Lipschitz isomorphic with distortion at most $K$, respectively.

A subset $\NN$ of a metric space $(\MM,d)$ is said to be \emph{uniformly separated} if
\[
\inf\{ d(x,y)\colon x,y\in\NN, x\not=y\}>0.
\]
A \emph{net} is a uniformly separated set $\NN$ with $\sup_{x\in\MM} d(x,\NN) <\infty$.

We use the symbol $\Nat_*$ to denote the set consisting of all non-negative integers, i.e., $\Nat_*=\Nat\cup\{0\}$.  Given $n\in\Nat$ we will put $\Nat_n=\Int\cap[0,n]$.

\section{The $p$-Banach space $\F_p(\Int^d)$ is not isomorphic to $\ell_p$}\label{sect:PpNotlp}
\noindent Once we conjecture that two quasi-Banach spaces are not isomorphic, the best strategy for substantiating our guess is to come up with a feature that tells apart them. We find this wished-for property within the theory of locally complemented subspaces and $\LL_p$-spaces developed by Kalton in \cite{Kalton1984}. A subspace $X$ of a quasi-Banach space $Y$ is said to be \emph{locally $K$-complemented} in $Y$ for some $K>0$, if for every finite-dimensional space $V\subset Y$ and every $\varepsilon>0$ there is $P\colon V\to X$ with
\begin{equation*}
\Vert P\Vert \le K \text{ and }\Vert P(x)-x\Vert\le\varepsilon\Vert x\Vert, \quad x\in V\cap X.
\end{equation*}
If the constant $K$ is irrelevant, we say that $X$ is locally complemented in $Y$.

A quasi-Banach space $X$ is a \emph{$\LL_p$-space,} $0<p<\infty$, if it is isomorphic to a locally complemented subspace of $L_p(\mu)$ for some measure $\mu$.
This notion was introduced in \cite{Kalton1984} with the aim of developing an attractive definition of a $\LL_p$-space for $0<p<1$ as it is not clear whether $L_p$($0<p<1$) would satisfy the analogue of the classical Lindenstrass-Pe{\l}czy\'{n}ski definition. For $p=1$, the above definition is equivalent to the classical one, while for $1<p<\infty$ the only exception to the equivalence is that, with the definition used in this paper, Hilbert spaces become $\LL_p$-spaces.

\begin{Lemma}\label{lem:slpdf}
Let $0<p<\infty$, let $X$ be a quasi-Banach space and let $(V_\alpha)_{\alpha\in A}$ be an increasing net consisting of finite-dimensional subspaces of $X$ with $\overline{\cup_{\alpha\in A} V_\alpha}=X$. Suppose that there is $K\in(0,\infty)$ such that for every $\alpha\in A$ and every $\varepsilon>0$ there is $S\colon V_\alpha\to \ell_p$ and $T \colon \ell_p \to X$ with
\[
\Vert T\Vert\,\Vert S\Vert \le K \text{ and }\Vert T(S(x))-x\Vert\le\varepsilon\Vert x\Vert, \quad x\in V_\alpha.
\]
Then $X$ is an $\LL_p$-space.
\end{Lemma}

\begin{proof}
Let $V$ be a finite-dimensional subspace of $X$ and let $\varepsilon>0$. Set
\[
\varepsilon_0:=\min\left\{1, \frac{\varepsilon}{\kappa+2\kappa^2} \right\}.
\]
A standard argument yields $\alpha\in A$ and $J\colon V\to V_\alpha$ with $\Vert J(x)-x\Vert\le\varepsilon_0\Vert x\Vert$ for all $x\in V$. By hypothesis, there are $S\colon V_\alpha\to\ell_p$ and $T\colon \ell_p\to X$ such that $\Vert T\Vert \,\Vert S\Vert \le K$ and $\Vert T(S(x))-x\Vert \le \varepsilon_0\Vert x\Vert$ for all $x\in V_\alpha$. We have
\[
\Vert J \Vert =\Vert J -\Id_V+\Id_V\Vert
\le \kappa (\varepsilon_0+1)\le 2\kappa.
\]
Hence $\Vert T\Vert \,\Vert S\circ J\Vert \le \Vert T\Vert \,\Vert S\Vert \,\Vert J\Vert \le 2\kappa K$.
If $x\in V$,
\begin{align*}
\Vert T(S(J(x)))-x \Vert
&\le \kappa( \Vert T(S(J(x)))-J(x)\Vert + \Vert J(x)-x\Vert)\\
&\le \kappa (\varepsilon_0 \Vert J(x)\Vert+ \varepsilon_0 \Vert x \Vert )\\
&\le \kappa \varepsilon_0 ( 2 \kappa \Vert x\Vert + \Vert x \Vert )
=\varepsilon \Vert x\Vert.
\end{align*}
Appealing to \cite{Kalton1984}*{Theorem 6.1} finishes the proof.
\end{proof}

It is known that complemented subspaces inherit the property of being $\LL_p$-spaces (see \cite{Kalton1984}*{Proposition 3.3}). For the sake of clarity, we state and prove the  quantitative version of this result, which is the one we will need.
\begin{Lemma}\label{lem:FDSLp}Let $X$ be an $\LL_p$-space and let $\lambda\in[1,\infty)$. There is a constant $K=K(X,\lambda)$ such that for every  quasi-Banach space $Y$ which is $\lambda$-complemented in $X$, every finite-dimensional subspace $V$ of $Y$, and every $\varepsilon>0$ there are bounded linear maps $S\colon V\to \ell_p$ and $T\colon\ell_p\to Y$ with $\Vert S \Vert \, \Vert T\Vert \le K$ and $\Vert T\circ S-\Id_V\Vert\le\varepsilon$.
\end{Lemma}

\begin{proof}Let $J\colon Y\to X$ and $P\colon X\to Y$ be such that $\Vert J \Vert \, \Vert P\Vert \le \lambda$ and $P\circ J=\Id_Y$. By \cite{Kalton1984}*{Theorem 6.1} there are $S\colon J(V)\to \ell_p$ and $T\colon\ell_p\to X$ with $\Vert S \Vert \, \Vert T\Vert \le K_0$ and $\Vert T\circ S-\Id_{J(V)}\Vert\le\varepsilon/\lambda$, where $K_0\in[1,\infty)$ depends only on $X$. We have
\[
\Vert S\circ J|_V\Vert\,\Vert P\circ T\Vert
\le \Vert S\Vert\, \Vert J\Vert\,\Vert P\Vert\, \Vert T\Vert
\le \lambda K_0
\] and
\[
\Vert P\circ T\circ S\circ J|_V-\Id_V\Vert =\Vert P\circ(T\circ S-\Id_{J(V)})\circ J\Vert \le \varepsilon.\qedhere
\]
\end{proof}

Given a quasi-Banach space $X$ and $0<q\le 1$ the $q$-Banach envelope of $X$ is a pair $(\Env[q]{X},E_{X,q})$, where $\Env[q]{X}$ is a $q$-Banach space and $E_{X,q}\colon X\to \Env[q]{X}$ is a linear contraction, defined by the following universal property: for every bounded linear map $T\colon X\to Y$, where $Y$ is a $q$-Banach space, there is a unique linear map $T'\colon \Env[q]{X}\to Y$ with $\Vert T'\Vert\le \Vert T\Vert$ and $T=T'\circ E_{X,q}$. If $X$ and $Y$ are quasi-Banach spaces and $T\colon X\to Y$ is linear and bounded, there is a unique bounded linear map $\Env[q]{T}\colon \Env[q]{X}\to \Env[q]{Y}$ such that $\Env[q]{T}\circ E_{X,q}=E_{Y,q}\circ T$. Moreover, $\Vert \Env[q]{T}\Vert\le \Vert T\Vert$. For background on envelopes, see e.g.\@ \cite{AABW2019}*{\S 9}.

Banach envelopes inherit BAP. For further reference, we record this result.

\begin{Lemma}\label{lem:BAPEnv}
Let $X$ be a quasi-Banach space with the BAP. Then $\Env[q]{X}$ has the BAP for $0<q\le 1$.
\end{Lemma}

Given $0<p\le 1$, a subset $\C$ of a vector space $V$ is said to be \textit{absolutely $p$-convex} if for any $x$ and $y\in \C$ and any scalars $\lambda$ and $\mu$ with $\ |\lambda|^p+|\mu|^p \le 1$ we have $\lambda \, x + \mu\, y \in \C$. We will denote by $\co_p(Z)$ the \textit{$p$-convex hull} of $Z\subset V$, i.e., the smallest absolutely $p$-convex set containing $Z$.

\begin{Lemma}\label{lem:compacthull}
Let $Z$ be a $q$-Banach space and $K\subset Z$ be relatively compact. Then the absolutely $q$-convex hull $\co_q(K)$ of $K$ is relatively compact.
\end{Lemma}

\begin{proof}
Since the map $(t,x)\mapsto tx$ is continuous and the unit sphere of the scalar field is compact, we can suppose that $t x\in K$ whenever $x\in K$ and $|t|=1$. If suffices to prove that $\co_q(K)$ possesses a finite $\varepsilon$-net for every $\varepsilon>0$. Let $\NN_0$ be a finite $(2^{-1/q}\varepsilon)$-net for $K$. Then
\[
\NN_1=\left\{\sum_{k=1}^m a_k x_k\colon x_k\in\NN_0,\, a_k\ge 0,\, \sum_{k=1}^m a_k^q\le 1\right\}
\]
is a $(2^{-1/q}\varepsilon)$-net for $\co_q(K)$

Enumerate $\NN_0=\{ y_j \colon 1\le j \le n\}$, and let $s\in\Nat$ be such that $ 2n \sup_j\Vert y_j\Vert^q \le \varepsilon^q 2^{sq}$. We will conclude the proof by showing that
\[
\left\{2^{-s} \sum_{j=1}^n b_j y_j\colon b_j\in\Nat_*,\,
\sum_{j=1}^n b_j^q\le 2^{sq}\right\}
\]
is a (finite) $(2^{-1/q}\varepsilon)$-net for $\NN_1$. Let $x\in\NN_1$. There is $(a_j)_{j=1}^n$ in $[0,\infty)$ such that $x=\sum_{j=1}^n a_j y_j$ and $\sum_{j=1}^n a_j^q\le 1$. For each $j=1$, \dots, $n$, let $b_j\in\Nat_*$ be such that $b_j\le 2^s a_j<b_j+1$. Then,
\[
\left\Vert x - 2^{-s} \sum_{j=1}^n b_j y_j\right\Vert^q\le
n 2^{-sq} \sup_j \Vert y_j\Vert^q \le \frac{\varepsilon^q}{2}.\qedhere
\]
\end{proof}

\begin{Theorem}\label{thm:compactEnv}
Let $T\colon X\to Y$ be a compact linear operator between quasi-Banach spaces $X$ and $Y$. Then the operator $\Env[q]{T}$ is compact for any $0<q\le 1$.
\end{Theorem}

\begin{proof}
Since the space of compact operators forms an ideal, we can assume that $Y$ is $q$-Banach, so that $\Env[q]{T}$ is a map from $\Env[q]{X}$ into $Y$ with $\Env[q]{T}\circ E_{q,X}=T$. Then, by construction,
\begin{align*}
\overline{\Env[q]{T}(B_{\Env[q]{X}})}
&= \overline{\Env[q]{T}(\co_q(E_{q,X}(B_X)))}\\
&= \overline{\co_q(\Env[q]{T}(E_{q,X}(B_X)))}\\
&=\overline{\co_q(T(B_X))},
\end{align*}
which is compact by Lemma~\ref{lem:compacthull}.
\end{proof}

Let $X$ be a subspace of a quasi-Banach space $Y$. We say that $X$ has the \emph{compact extension property} (CEP for short) in $Y$ if every compact operator $T\colon X\to Z$, where $Z$ is a quasi-Banach space, extends to a compact operator $\widetilde T\colon Y\to Z$. Let $0<q\le 1$. We say that $X$ has the \emph{compact extension property for $q$-Banach spaces} ($q$-CEP for short) in $Y$ if the compact extension property holds when $Z$ is a $q$-Banach space. If, moreover, we can ensure that $\Vert \widetilde T \Vert\le K \Vert T\Vert$ for some constant $K\in[1,\infty)$ (depending on $X$, $Y$ and $q$), we say that $X$ has the $q$-CEP in $Y$ with constant $K$. Notice that, by the Aoki-Rolewicz theorem, $X$ has the CEP in $Y$ if and only if $X$ has the $q$-CEP in $Y$ for every $0<q\le 1$. We have the following results in this respect.

\begin{Theorem}\label{thm:Lindens}
Let $X$ be a subspace of a quasi-Banach space $Y$. Suppose that $X$ has the $q$-CEP in $Y$ for some $0<q\le 1$. Then there is a constant $K\in[1,\infty)$ for which $X$ has the $q$-CEP in $Y$ with constant $K$.
\end{Theorem}

\begin{proof}
Although the proof of \cite{Lind1964}*{Theorem 2.2} was done for Banach spaces, the arguments therein apply to our case without any important modifications.
\end{proof}

\begin{Theorem}\label{thm:CEPvsLC}
Let $0<q\le 1$ and $Y$ be a $q$-Banach space with the BAP. Let $X$ be a closed subspace of $Y$. The following are equivalent.
\begin{enumerate}[label={(\roman*)}, leftmargin=*, widest=iii]
\item\label{thm:CEPvsLCitem1} $X$ is locally complemented in $Y$.
\item\label{thm:CEPvsLCitem2} $X$ has the BAP and the CEP in $Y$.
\item\label{thm:CEPvsLCitem3} $X$ has the BAP and the $q$-CEP in $Y$.
\end{enumerate}
\end{Theorem}

\begin{proof}
Although \cite{Kalton1984}*{Theorem 5.1} only establishes the equivalence between \ref{thm:CEPvsLCitem1} and \ref{thm:CEPvsLCitem2}, the very same proof gives that \ref{thm:CEPvsLCitem3} implies \ref{thm:CEPvsLCitem1}. So, taking into account Theorem~\ref{thm:Lindens}, \ref{thm:CEPvsLCitem3} is equivalent to \ref{thm:CEPvsLCitem1} and \ref{thm:CEPvsLCitem2}.
\end{proof}

\begin{Theorem}\label{lem:lcenv}
Let $X$ be a subspace of a quasi-Banach space $Y$. Denote by $J$ the inclusion of $X$ into $Y$. Suppose that $X$ has the BAP and that $X$ has the $q$-CEP in $Y$ for some $0<q\le 1$. Then $\Env[q]{J}$ is an isomorphic embedding and $\Env[q]{J}(\Env[q]{X})$ has the $q$-CEP in $\Env[q]{Y}$.
\end{Theorem}

\begin{proof}
Use Theorem~\ref{thm:Lindens} to pick $K$ such that $X$ has the $q$-CEP in $Y$ with constant $K$. Let $Z$ be a $q$-Banach space and let $T\colon\Env[q]{X}\to Z$ be a compact operator. Since $T\circ E_{X,q}$ is also compact, the compact extension property and Theorem~\ref{thm:compactEnv} yield a compact operator $S\colon \Env[q]{Y}\to Z$ with $\Vert S \Vert \le K \Vert T\Vert$ and such that the diagram
\[
\xymatrix{Y \ar[r]^{E_{Y,q}} & \Env[q]{Y} \ar[dr]^{S} & \\
X \ar[r]_{E_{X,q}} \ar[u]^{J} & \Env[q]{X} \ar[r]_{T} & Z}
\]
commutes. Since $E_{X,q}(X)$ is dense in $\Env[q]{X}$, merging this commutative diagram with
\[
\xymatrix{Y \ar[r]^{E_{Y,q}} & \Env[q]{Y} \\
X \ar[r]_{E_{X,q}} \ar[u]^{J} & \Env[q]{X} \ar[u]_{\Env[q]{J}}}
\]
yields that the diagram
\[
\xymatrix{\Env[q]{Y} \ar[dr]^{S} & \\
\Env[q]{X} \ar[r]_{T} \ar[u]^{\Env[q]{J}} & Z}
\]
commutes. It remains to show that $\Env[q]{J}$ is an isomorphic embedding. To that end, use Lemma~\ref{lem:BAPEnv} to pick $C$ such that $\Env[q]{X}$ has that BAP with constant $C$. Let $x\in \Env[q]{X}$ and $\varepsilon>0$. There is a linear operator $T\colon \Env[q]{X}\to \Env[q]{X}$ with finite-dimensional range such that $\Vert x-T(x)\Vert \le \varepsilon$ and $\Vert T\Vert \le C$. Since $T$ is compact, there is $S\colon \Env[q]{Y}\to \Env[q]{X}$ such $S\circ \Env[q]{J}=T$ and $\Vert S\Vert \le C K$. We have
\[
\Vert x \Vert^q \le \Vert x-T(x)\Vert^q + \Vert S (\Env[q]{J}(x))\Vert^q
\le \varepsilon^q+ C^q K^q \Vert \Env[q]{J}(x)\Vert^q.
\]
Letting $\varepsilon$ tend to zero we obtain $\Vert x \Vert \le CK \Vert \Env[q]{J}(x)\Vert$.
\end{proof}

\begin{Theorem}\label{thm:EnvSLpIsSLq}
Let $0<p\le q \le 1$ and $X$ be a separable $\LL_p$-space with the BAP. Then $\Env[q]{X}$ is isomorphic to a locally complemented subspace of $\ell_q$ and has a Schauder basis.
\end{Theorem}

\begin{proof}
By \cite{Kalton1984}*{Theorem 6.4}, $X$ is isomorphic to a locally complemented subspace $Y$ of $\ell_p$, and $Y$ has a Schauder basis. By Theorem~\ref{thm:CEPvsLC}, $Y$ has CEP in $\ell_p$. Hence, by Theorem~\ref{lem:lcenv}, $\Env[q]{X}$ is isomorphic to a subspace $Z$ of $\ell_q$ and has the $q$-CEP in $\ell_q$. As it is clear that $\Env[q]{Y}$ (and so  $\Env[q]{X}$) has a Schauder basis, applying once again Theorem~\ref{thm:CEPvsLC} completes the proof.
\end{proof}

The following straightforward consequence of Theorem~\ref{thm:EnvSLpIsSLq} partially solves \cite{AAW2020}*{Question 4.18}.
\begin{Corollary}\label{cor:EnvLpisLq}
Let $0<p\le q \le 1$ and $X$ be a separable $\LL_p$-space with the BAP. Then $\Env[q]{X}$ is a $\LL_q$-space.
\end{Corollary}

We are ready to prove the main results of this section.
\begin{Theorem}\label{thm:FpMNotSLp}Let $\MM$ be a $p$-metric space, $0<p<1$. Suppose that there is a constant $C$ such that $\Nat_n$ is a Lipschitz retract of $\MM$ with constant $C$ for all $n\in\Nat$. Then $\F_p(\MM)$ is not an $\LL_p$-space. In particular, $\F_p(\MM)\not\simeq\ell_p$.
\end{Theorem}

\begin{Theorem}\label{thm:FpMNotSLpB}
Let $\MM$ be a metric space. Suppose that there is a constant $K$ such that, for all $n\in\Nat$, $\Nat_n$ Lipschitz-isomorphically embeds in $\MM$ with distortion at most $K$. Then $\F_p(\MM)$ is not an $\LL_p$-space. In particular, $\F_p(\MM)\not\simeq\ell_p$.
\end{Theorem}

\begin{proof}[Proof of Theorems~\ref{thm:FpMNotSLp} and \ref{thm:FpMNotSLpB}]
Since $\Rea$ is a doubling metric space, there is $D\ge 1$ such that every subset of $\Rea$, in particular $\Nat_n$ for all $n\in\Nat$, is a doubling metric space with doubling constant $D$. Thus, under the assumptions in Theorem~\ref{thm:FpMNotSLpB}, applying \cite{AACD2020}*{Theorem 5.1} yields a constant $K$ such that $\F_p(\Nat_n)$ is $K$-complemented in $\F_p(\MM)$ for all $n\in\Nat$.
By \cite{AACD2018}*{Lemma 4.19}, this holds under the assumptions in Theorem~\ref{thm:FpMNotSLp} as well. Since
\[
\Dy_k=\{ x\in [-k,k] \colon 2^k x\in\Int\}
\]
is a doubling metric space with constant $D$ for all $k\in\Nat$, applying again \cite{AACD2020}*{Theorem 5.1} yields a constant $C$ such that the linearization of the inclusion of $\Dy_k$ into $\Rea$ is a $C$-isomorphic embedding for all $k\in\Nat$. Taking into account that $\Dy_k$ is Lipschitz isomorphic to $\Nat_{k 2^{k+1}}$ with distortion $1$ we infer the existence of a constant $K_1$ such that the finite-dimensional subspace
\[
V_k:=\spn( \delta_\Rea(x) \colon x \in \Dy_k)
\]
of $\F_p(\Rea)$ is $K_1$-complemented in $\F_p(\MM)$ for all $k\in\Nat$. 

Suppose by contradiction that $\F_p(\MM)$ is an $\LL_p$-space. Then, by Lemma~\ref{lem:FDSLp}, there is a constant $K_2$ such that for all $k\in\Nat$ and $\varepsilon>0$ there are
linear bounded maps $S \colon V_k\to \ell_p$ and $T \colon\ell_p\to V_k\subset\F_p(\Rea)$ with $\Vert S \Vert \, \Vert T \Vert \le K_2$ and $\Vert T\circ S - \Id_{V_k} \Vert \le \varepsilon$.
Since the set $\Dy$ consisting of all dyadic rationals is dense in $\Rea$,
\[
\cup_{k=1}^\infty V_k=\spn(\delta_\Rea(x) \colon x\in\Dy)
\]
is a dense subspace of $\F_p(\Rea)$. Applying Lemma~\ref{lem:slpdf} yields that $\F_p(\Rea)$ is an $\LL_p$-space. Since $\F_p(\Rea)$ has the BAP (see \ref{three:piproperty}) combining Theorem~\ref{thm:EnvSLpIsSLq} with  \cite{AACD2018}*{Proposition 4.20} yields that $\F(\Rea)$ is isomorphic to a subspace of $\ell_1$. Using that $\F(\Rea)$ is isometric to $L_1(\Rea)$ and that $\ell_2$ is a subspace of $L_1(\Rea)$, we obtain that $\ell_2$ is isomorphic to a subspace of $\ell_1$, an absurdity.

For the last part of the statements, we just need to note that $\ell_p$ is an $\LL_p$-space.
\end{proof}

\begin{Corollary}\label{thm:notEllp}
Let $\MM$ be a metric space containing a subset which is Lipchitz isomorphic either to $[0,1]$ or to $\Nat$. Then $\F_p(\MM)$ is not a $\LL_p$-space for any $0<p<1$.
\end{Corollary}

\begin{proof}
Just notice that $\{ k/n \colon k \in\Nat_n\}\subset[0,1]$ is Lipschitz isomorphic to $\Nat_n$ with distortion $1$, and apply Theorem~\ref{thm:FpMNotSLpB}.
\end{proof}

\begin{Corollary}
Let $X$ be a $p$-Banach space ($0<p<1$) with nontrivial dual. Then $\F_p(X)$ is not a $\LL_p$-space.
\end{Corollary}

\begin{proof}By Corollary~\ref{thm:notEllp}, $\F_p(\Rea)$ is not a $\LL_p$-space.  Since, by assumption, $\Rea$ is a complemented subspace of $X$, by \cite{AACD2018}*{Lemma 4.19}  it follows that $\F_p(\Rea)$ is a complemented subspace of 
$\F_p(X)$. Consequently, $\F_p(X)$ is a not an $\LL_p$-space either.
\end{proof}

\section{Schauder bases in $\F_p([0,1]^d)$ and $\F_p(\Rea^d)$}\label{Sect:Basis}
\noindent
The basic idea for building Schauder bases for $\F_p([0,1]^d)$ and $\F_p(\Rea^d)$ comes, on one hand, from \cite{LP2013}, where the authors present a canonical way of extending linearly Lipschitz functions on $d$-dimensional hypercubes, and on the other hand from \cite{AACD2019}, where a method for building linear projections on $\F_p([0,1])$ is given.

Fix $d\in\Nat$. Given $R>0$ and $w\in\Int^d$, we denote by $Q^d_{w,R}$, the hypercube having edge-length $R$ and with vertices in the points
\[
V^d_{w,R}=\{ R w + R\varepsilon \colon \varepsilon\in\{0,1\}^d\},
\]
That is, if $w=(w_i)_{i=1}^d$,
\[
Q^d_{w,R}=\co[V^d_{w,R}]=\prod_{i=1}^d [R w_i, Rw_i+R].
\]
For $R>0$ fixed, the set of hypercubes
\[
\QQ^d_R= \{ Q^d_{w,R} \colon w\in\Int^d\}
\]
tessellates the space $\Rea^d$. If $Q\in\QQ^d_R$, we denote by $\VV(Q)$ its set of vertices, that is $\VV(Q^d_{w,R})=V^d_{w,R}$ for every $w\in\Int^d$. We have
\[
\bigcup_{Q\in \QQ^d_R} \VV(Q)=\VV^d_R:=\{ R w \colon w\in\Int^d\}.
\]

We shall define a fuzzy pull back of $\Rea^d$ into the set of vertices $\VV^d_R$. Given $x\in[0,1]$ and $w\in\Int$ we set
\[
x^{(w)}
=\begin{cases} x & \text{ if } w=1, \\ 1-x & \text{ if } w=0, \\ 0 & \text{ if } w\in\Int\setminus\{0,1\}.\end{cases}
\]
Given $x=(x_i)_{i=1}^d\in[0,1]^d$ and $w=(w_i)_{i=1}^d\in\Int^d$ we put
\[
x^{(w)}=\prod_{i=1}^d x_i^{(w_i)}.
\]

\begin{Lemma}\label{lem:FPB}
Let $d\in\Nat$ and $R>0$. There is a mapping
\[
\Lambda=(\Lambda(v,\cdot))_{v\in\VV^d_R} \colon \Rea^d \to [0,1]^{\VV^d_R}
\]
such that $\Lambda(Ru,Rw+Rx) = x^{(u-w)}$ for every $x\in [0,1]^d$ and $u,w\in\Int^d$.
\end{Lemma}

\begin{proof}
Since the function $x\mapsto Rw+Rx$ maps $[0,1]^d$ onto $Q^d_{w,R}$, if such a function $\Lambda$ exists, it is unique. By dilation, it suffices to consider the case $R=1$. If $\Lambda$ is as desired in the one-dimensional case, then
\[
\Lambda(u,\cdots)=\Lambda(u_1,\cdot)\otimes \cdots \otimes\Lambda(u_i,\cdot)\otimes \cdots \otimes \Lambda(u_d,\cdot), \quad u=(u_i)_{i=1}^d\in\Int^d,
\]
is as desired in the $d$-dimensional case. Hence, we can also assume that $d=1$. To prove the result in this particular case, we must check that, given $w\in\Int$, the function given for $x\in\Rea$ by
\[
x\mapsto (x-w)^{(u-w)} \text{ if } u\in\Int \text{ and } w\le x \le w+1
\]
is well-defined. Suppose that $w\le x \le w+1$ and $v\le x \le v+1$ with $v$, $w\in\Int$. Assume without lost of generality that $v<w$. Then $x=w=v+1$. Since $x-w=0$, we have $ (x-w)^{(u-w)}=0$ unless $u-w=0$, in which case $ (x-w)^{(u-w)}=1$. Since $x-v=1$, we have $(x-v)^{(u-v)}=0$ unless $u-v=1$, in which case $(x-v)^{(u-v)}=1$. Since $u-w=u-v-1$ for every $u\in\Int$, we are done.
\end{proof}

\begin{Definition}\label{def:FPB}
Given $R>0$ and $d\in\Nat$, we define
\[\Lambda_R^d=(\Lambda_{R}^d(v,\cdot))_{v\in\VV^d_R}\]
as the function provided by Lemma~\ref{lem:FPB}.
\end{Definition}

If $d=1$ we simply put $\Lambda_R=\Lambda^1_R$ and $\VV_R=\VV^1_R$. Given a finite set $A\subset \Nat$ we can carry out the above construction replacing the set $\{1,\dots,d\}$ with the set $A$. We will denote by $\VV^A_R$ the corresponding set of vertices and by $\VV^A_R$ the corresponding function defined as in Definition~\ref{def:FPB}.

Let us give an auxiliary lemma followed by some properties of the function $\Lambda_R^d$.

\begin{Lemma}\label{lem:ineq}
Let $x=(x_i)_{i=1}^d$, $y=(y_i)_{i=1}^d\in[0,1]^d$. Then
\[
\left| \prod_{i=1}^d x_i - \prod_{i=1}^d y_i\right|\le \vert x-y\vert_1.
\]
\end{Lemma}

\begin{proof}
We proceed by induction on $d$. For $d=1$ the result is obvious. Assume that $d\in\Nat$ and that the result holds for $d-1$. Then
\begin{align*}
\left |\prod_{i=1}^d x_i - \prod_{i=1}^d y_i\right|
&\le \left|\prod_{i=1}^d x_i - y_d \prod_{i=1}^{d-1} x_i\right| + \left| y_d \prod_{i=1}^{d-1} x_i - \prod_{i=1}^d y_i\right| \\
&= |x_d-y_d| \prod_{i=1}^{d-1} x_i + y_d \left|\prod_{i=1}^{d-1} x_i - \prod_{i=1}^{d-1} y_i\right| \\
&\le |x_d-y_d| + \sum_{i=1}^{d-1} |x_i-y_i| =\vert x-y\vert_1.\qedhere
\end{align*}
\end{proof}

\begin{Lemma}\label{prop:FPB}
Let $d\in\Nat$ and $S\ge R>0$ with $S/R\in\Int$. We have:
\begin{enumerate}[label={(\roman*)}, leftmargin=*,widest=iii]
\item\label{prop:FPB:1} $\Lambda^d_R(v,x)=0$ if $x\in Q\in\QQ^d_R$ and $v\notin \VV(Q)$.
\item\label{prop:FPB:2} $\Lambda^d_R(v,u)=\delta_{u,v}$ for every $u$, $v\in \VV^d_R$.
\item\label{prop:FPB:3} $\sum_{v\in\VV_R} \Lambda^d_R(v,x)=1$ for every $x\in\Rea^d$.
\item\label{prop:FPB:4} If there is $Q\in\QQ^d_R$ such that $x$, $y\in Q$, then
\[|\Lambda^d_R(v,x)-\Lambda^d_R(v,y)|\le R^{-1} \vert x-y\vert_1\] for every $v\in\VV^d_R$.
\item\label{prop:FPB:6}Let $(A,B)$ be a partition of $\{1,\dots,d\}$. Then
\[
\Lambda^d_R((u,v),(x,y))=\Lambda^A_R(u,x)\Lambda^B_R(v,y)
\]
for every $u\in\VV^A_R$, $v\in\VV^B_R$, $x\in\Rea^A$ and $y\in\Rea^B$.
\item\label{prop:FPB:5} $\Lambda^d_S(v,x)=\sum_{u\in\VV^d_R} \Lambda^d_S(v,u)\Lambda^d_R(u,x)$ for every $x\in\Rea^d$ and $v\in\VV^d_S$.
\end{enumerate}
\end{Lemma}

\begin{proof}
\ref{prop:FPB:1} is clear from the definition. \ref{prop:FPB:2} follows from the equality $0^{(0)}=1$. A straightforward induction on $d$ yields
\[
\sum_{w\in\Int^d} x^{(w)}=1, \quad x\in[0,1]^d,
\]
and \ref{prop:FPB:3} is clear from this identity. \ref{prop:FPB:4} is a consequence of Lemma~\ref{lem:ineq}. \ref{prop:FPB:6} is clear from the definition. In light of (v), in order to prove \ref{prop:FPB:5} it suffices to consider the case $d=1$. Given $x\in\Rea$ there are $u_0$, $u_1\in\VV_S$ and $v_0$, $v_1\in\VV_R$ with $u_0\le v_0 \le x<v_1\le u_1$, and we have $v_1=v_0+R$ and $u_1=u_0+S$. Suppose that $u\in\VV_S\setminus\{u_0,u_1\}$. Then $\Lambda_S(u,v)=0$ for $v\in\{v_0,v_1\} $. Since $\Lambda_R(v,x)=0$ for $v\in\VV_R\setminus\{ v_0,v_1\} $ we have
\[
\Gamma(u,x):=\sum_{v\in\VV_R} \Lambda_S(u,v)\Lambda_R(v,x)=0=\Lambda_S(u,x).
\]
Hence, considering also the symmetry $x\mapsto -x$, if suffices to prove the result in the case when $u=u_1$. We have
\begin{align*}
\Gamma(u_1,x)
&=\Lambda_S(u_1,v_0)\Lambda_R(v_0,x) + \Lambda_S(u_1,v_1)\Lambda_R(v_1,x)\\
&=\eta(x):=\frac{v_0-u_0}{S} \frac{v_1-x}{R} + \frac{v_1-u_0}{S} \frac{x-v_0}{R}.
\end{align*}
Since $\eta(u_0)=0$ we have $\eta(y)=\gamma(y-u_0)$ for all $y\in\Rea$, where
\[
\gamma=\left(-\frac{v_0-u_0}{SR}+\frac{v_1-u_0}{SR} \right)=\frac{1}{S}.
\]
Since $\Lambda_S(u_1,x)=(x-u_0)/S$ we are done.
\end{proof}

Although the previous auxiliary results are stated in terms of the $\ell_1$-norm, in this section we will consider $\Rea^d$ and its subsets equipped with the supremum norm $\|\cdot\|_\infty$.
\begin{Theorem}\label{thm:LipFPB}
Let $d\in\Nat$ and $0<p\le 1$. There is a constant $C=C(p,d)$ such that for every $R>0$ and every $\RR\subset \QQ^d_R$, if we denote $K=\cup_{Q\in\RR} Q$ and $V=\cup_{Q\in\RR} \VV(Q)$, and we choose an arbitrary point of $V$ as base point of both metric spaces, there is a $C$-Lipschitz map $r=r_{K,V}\colon K\to \F_p(V)$ such that
\[
r(x) = \sum_{v\in V} \Lambda^d_R(v,x) \, \delta_V(v), \quad x\in K.
\]
\end{Theorem}

\begin{proof}Let $Q\in\RR$ and $x$, $y\in Q$. Pick $u\in \VV(Q)$. By Lemma~\ref{prop:FPB},
\begin{align*}
\Vert r(x)-r(y)\Vert^p
&= R \left\Vert \sum_{v\in \VV(Q)\setminus\{ u\} } (\Lambda^d_R(v,x) - \Lambda^d_R(v,y)) \frac{\delta_V(v)-\delta_V(u)}{\vert v-u\vert_\infty} \right\Vert^p \\
&\le R \sum_{v\in \VV(Q)\setminus\{ u\} } | \Lambda^d_R(v,x) - \Lambda^d_R(v,y) |^p \\
&\le (2^d-1) \vert x-y\vert_1^p.
\end{align*}
Let $x=(x_i)_{i=1}^d\in K$ and $y=(y_i)_{i=1}^d\in K$. Pick $u=(u_i)_{i=1}^d$ and $w=(w_i)_{i=1}^d\in\Int^d$ such that $x\in Q^d_{R,u}$ and $y\in Q^d_{R,w}$. Define
\[
F=\{i\in\{1,\dots,d\} \colon u_i=w_i\}.
\]
For each $i\in G=\{1,\dots,d\}\setminus F$ there is $m_i\in\{u_i,u_i+1\}$ and $n_i\in\{w_i,w_i+1\}$ such that
\[
|y_i-x_i|=|y_i-Rn_i|+|Rn_i-Rm_i|+|Rm_i-x_i|.
\]
Suppose that $n_i=m_i$ for every $i\in G$. Define $z=(z_i)_{i=1}^d$ by
\[
z_i=\begin{cases} x_i & \text{ if } i\in F, \\ Rn_i=Rm_i & \text{ if } i\in G. \end{cases}
\]
We have $z\in Q^d_{R,u} \cap Q^d_{R,w}$. Consequently,
\begin{align*}
\Vert r(x)-r(y)\Vert^p
&\le \Vert r(x)-r(z)\Vert^p+\Vert r(z)-r(y)\Vert^p\\
&\le (2^d-1)( \vert x-z\vert_1^p+\vert z-y\vert_1^p)\\
&\le 2^{1-p} (2^d-1) (\vert x-z\vert_1+\vert z-y\vert_1)^p\\
&= 2^{1-p} (2^d-1) \vert x-y\vert_1^p.
\end{align*}
Suppose that $m_i\not=n_i$ for some $i$. Define $x'=(x_i')_{i=1}^d$ and $y'=(y_i')_{i=1}^d$ by
\[
x_i'=\begin{cases} Ru_i=Rw_i & \text{ if } i\in F, \\ Rm_i & \text{ if } i\in G, \end{cases}
\quad
y_i'=\begin{cases} Ru_i=Rw_i & \text{ if } i\in F, \\ R n_i & \text{ if } i\in G.
\end{cases}
\]
We have $x'\in V^d_{R,u}$, $y'\in V^d_{R,w}$, and $1\le \mu:=\max_{i\in G} |m_i-n_i|$. Hence
\begin{align*}
\Vert r(x)&-r(y)\Vert^p\\
&\le \Vert r(x)-r(x')\Vert^p+ \Vert r(x')-r(y')\Vert^p+ \Vert r(y')-r(y)\Vert^p\\
&\le (2^d-1)\vert x-x'\vert_1^p+\vert x'-y'\vert_\infty^p + (2^d-1)\vert y-y'\vert_1^p\\
&\le 2 (2^d-1) d R^p + \mu^p R^p \\
&\le (1 + 2 d (2^d-1) ) \mu^p R^p\\
&\le (1 + 2 d (2^d-1) ) \vert x-y\vert_\infty.
\end{align*}
This way the result holds with $C(p,d)= (1 + 2 d (2^d-1) )^{1/p}$.
\end{proof}

Next, we show that $\F_p(\Rea^d)$ has a Schauder basis for every $d\in\Nat$.
\begin{Lemma}\label{lem:17}
Let $\MM$ and $\NN$ be quasi-metric spaces. For $i=1$, $2$, let $f_i\colon\MM_i\subset\MM\to \NN$ be a Lipschitz function. Assume that
\begin{enumerate}[label={(\roman*)}, leftmargin=*,widest=iii]
\item $\MM=\MM_1\cup \MM_2$,
\item $f_1|_{\MM_1\cap \MM_2}=f_2|_{\MM_1\cap \MM_2}$, and
\item\label{it:17:3} There is a constant $C$ such that $d(x_1,\MM_1\cap \MM_2)\le Cd(x_1,x_2)$ for every $(x_1,x_2)\in\MM_1\times\MM_2$.
\end{enumerate}
Then, the map $f\colon\MM\to\NN$ defined by $f|_{\MM_i}=f_i$ for $i=1$, $2$ is Lipschitz. Moreover, if $k_\NN$ and $k_\MM$ are the quasi-metric constanst of $\NN$ and $\MM$, respectively, we have
\[
\Lip(f) \le k_\NN(C+k_\MM + Ck_\MM)\max_{i=1,2} \Lip(f_i).
\]
\end{Lemma}

\begin{proof}
Put $L=\max_{i=1,2} \Lip(f_i)$. Let $(x_1,x_2)\in\MM_1\times\MM_2$ and pick $x\in\MM_1\cap \MM_2$. Since $f_1(x)=f_2(x)$ we have
\begin{align*}
d_\NN(f_1(x_1), f_2(x_2))
&\le k_\NN ( d_\NN(f_1(x_1), f_1(x))+d_\NN(f_2(x), f_2(x_2)))\\
&\le k_\NN L ( d_\MM(x_1,x)+ d_\MM(x,x_2))\\
&\le k_\NN L ( (1+k_\MM) d_\MM(x_1,x)+ k_\MM d_\MM(x_1,x_2)).
\end{align*}
Since the element $x$ can be chosen so that $d_\MM(x_1,x)$ is arbitrarily close to $d_\MM(x_1,\MM_1\cap\MM_2)$, using the assumption \ref{it:17:3} yields the desired result.
\end{proof}

The following lemma exhibits a situation in which Lemma~\ref{lem:17} is useful that will occur several times throughout this section,.
\begin{Lemma}\label{lem:18}
Let $\MM$ and $\NN$ be quasi-metric spaces. For $i=1$, $2$, let $f_i\colon\MM_i\subset\MM\to \NN$ be a Lipschitz function. Assume that $\MM=\MM_1\cup \MM_2$ and that $f_1|_{\MM_1\cap \MM_2}=f_2|_{\MM_1\cap \MM_2}$. Suppose that there are constants $\lambda>0$ and $C\ge 1$ such that
\begin{enumerate}[label={(\roman*)},widest=ii]
\item $\MM$ is $\lambda$-separated, and
\item $d(x_1,\MM_1\cap \MM_2)\le C$ for every $x_1\in\MM_1$.
\end{enumerate}
Then, the map $f\colon\MM\to\NN$ defined by $f|_{\MM_i}=f_i$ for $i=1$, $2$ is Lipschitz. Moreover, if $k_\NN$ and $k_\MM$ are the quasi-metric constanst of $\NN$ and $\MM$, respectively,
\[
\Lip( f) \le k_\NN \left( \frac{C}{\lambda} ( 1+k_\MM) +k_\MM \right) \max_{i=1,2} \Lip (f_i).
\]
\end{Lemma}

\begin{proof}Let $(x_1,x_2)\in\MM_1\times\MM_2$ with $x_1\not=x_2$. We have
\[
d(x_1,\MM_1\cap \MM_2)\le \frac{C}{\lambda} d(x_1,x_2).
\]
Hence, the result follows from Lemma~\ref{lem:17}.
\end{proof}

\begin{Theorem}[cf. \cite{AACD2019}*{Theorem 5.7}]\label{thm:basis01d}
Let $0<p\le 1$ and $d\in\Nat$. Then $\F_p([0,1]^d)$ has a Schauder basis. In fact, if $V_n=[0,1]^d \cap 2^{-n}\Int^d$ for all $n\in\Nat_*$ and $V_{-1}=\{0\}$, and we put $\alpha(x)=n$ if $x\in V_n\setminus V_{n-1}$, then, any arrangement $(f(x_j))_{j=1}^\infty$ of the family
\[
f(x)=\delta_{[0,1]^d}(x)-\sum_{v\in V_{n-1}} \Lambda^d_{2^{-n+1}}(v,x) \delta_{[0,1]^d}(v) \quad n\in\Nat_*, \, x\in V_n\setminus V_{n-1}
\]
such that $(\alpha(x_j))_{j=1}^\infty$ is non-decreasing is a Schauder basis of $\F_p([0,1]^d)$.
\end{Theorem}

\begin{proof}
Set $\delta=\delta_{[0,1]^d}$ and $\delta_n=\delta_{V_n}$ for $n\in\Nat_*$.
By Theorem~\ref{thm:LipFPB}, there exist a constant $C$ and linear maps $T_n\colon \F_p([0,1]^d)\to \F_p(V_n)$ such that $\Vert T_n\Vert\le C$ and
\[
T_n(\delta(x)) = \sum_{v\in V_n} \Lambda^d_{2^{-n}}(v,x) \, \delta_{n}(v), \quad x\in [0,1]^d.
\]
Let $m$, $n\in\Nat_*$ with $m\le n$. Since $V_m\subset V_n$, we can consider the canonical map $L_{m,n}\colon \F_p(V_m)\to\F_p(V_n)$ associated to the inclusion. Consider also the canonical map $L_n\colon \F_p(V_n)\to\F_p([0,1]^d)$ associated to the inclusion of $V_n$ into $[0,1]^d$. Applying Proposition~\ref{prop:FPB}~\ref{prop:FPB:5} yields
\begin{itemize}[leftmargin=*]
\item $T_m\circ L_{n}\circ T_n=T_m$ and $T_n \circ L_m\circ T_m=L_{m,n}\circ T_m$.
\end{itemize}
Moreover, since $\cup_{n=0}^\infty V_n$ is dense in $[0,1]^d$,
\begin{itemize}[leftmargin=*]
\item if $X_n= L_n(T_n(\F_p([0,1]^d)))$, $\cup_{n=0}^\infty X_n$ is dense is $\F_p([0,1]^d)$.
\end{itemize}
Since, with the convention $L_{-1}\circ T_{-1}=0$,
\[
Y_n=\{ x -L_{n-1}(T_{n-1}(x)) \colon x\in X_n\},
\]
and $X_n=\spn\{\delta(x) \colon x \in V_n\}$, the family of nonzero vectors
\[
\BB_n:=(f (x))_{x\in V_n\setminus V_{n-1}}
\]
generates the space $Y_n$. We shall prove that $\BB_n$ is an unconditional basis of $Y_n$ with uniformly bounded unconditional basis constant. Let $F\subset V_n\setminus V_{n-1}$. We define $r_{n,F}\colon V_{n} \to \F_p([0,1]^d)$ by
\[
r_{n,F}(x)=\begin{cases} L_{n-1}\circ T_{n-1}(\delta(x)) &\text{ if } x\in V_{n}\setminus F, \\ \delta(x) &\text{ if } x\in F. \end{cases}
\]
If $z\in V_{n-1}$, then $L_{n-1}( T_{n-1}(\delta(z)))=L_{n-1}(\delta_{n-1}(z))=\delta(z)$. For every $x\in V_n$ there is $z\in V_{n-1}$ such that $\vert x-z\vert_\infty=2^{-n}$, and $V_n$ is $2^{-n}$-separated. By Lemma~\ref{lem:18}, $r_{n,F}$ is $C_1$-Lipschitz with $C_1=(1+2^{1/p})2^{1/p-1}C$. We infer that there is $U_{n,F}\colon \F_p(V_{n})\to \F_p([0,1]^d)$ such that $\Vert U_{n,F}\Vert\le C_1$ and $U_{n,F}\circ \delta_{n}=r_{n,F}$. Put $Q_{n,F}=U_{n,F}\circ T_n|_{Y_n}$. We have $\Vert Q_{n,F}\Vert \le C C_1$ and
\[
Q_{n,F}(f(x))=\begin{cases}
f(x) & \text{ if } x\in F, \\
0 & \text{ if } x\in (V_n\setminus V_{n-1})\setminus F.
\end{cases}
\]
Thus, the mappings $(Q_{n,F})_{F\subset V_n\setminus V_{n-1}}$ are commuting projections satisfying $Q_{n,F}(Y_n)=\spn\{f(x)\colon x\in F\}$ and $Q_{n,F}^{-1}(0)=\spn\{f(x)\colon x\in (V_n\setminus V_{n-1})\setminus F\}$, which implies that $\B_n$ is unconditional basis of $Y_n$ with basis constant at most $CC_1$.
\end{proof}

We have explicitly constructed a Schauder basis of $\F_p([0,1]^d)$. Using the isomorphism $\F_p([0,1]^d)\simeq\F_p(\Rea^d)$ we infer that $\F_p(\Rea^d)$ has a Schauder basis, but now our proof is not constructive. So, it is worth to mention that an argument slightly more involved than the one we have used to prove Theorem~\ref{thm:basis01d} yields an explicit Schauder basis of $\F_p(\Rea^d)$.

\begin{Theorem}\label{thm:basisRd}Let $0<p\le 1$ and $d\in\Nat$. Given an increasing sequence of natural numbers $(k_n)_{n=-1}^\infty$, put
\begin{align*}
V_n&=\{x \in \Rea^d \colon \vert x \vert_\infty \le k_n\} \cap 2^{-n}\Int^d, \text{ and}\\
W_n&=\{x \in \Rea^d \colon \vert x \vert_\infty \le k_{n-1}\} \cap 2^{-n}\Int^d\\
\end{align*}
if $n\in\Nat_*$, and set $V_{-1}=\{0\}$. Define
\[
s_n(x)=\left( \min \left\{ -2^{-n}+ \vert x \vert_\infty, |x_i| \right\}\sgn(x_i)\right)_{i=1}^d, \, x=(x_i)_{i=1}^d\in V_n\setminus W_n,
\]
and put $\eta(x)=(n,\vert x\vert_\infty-k_{n-1})$ if $x\in V_n\setminus W_n$ and $\eta(x)=(n,0)$ if $x\in W_n\setminus V_{n-1}$. Define $f(x)$ for $x \in \cup_{n=0}^\infty V_n$ by
\[
f(x)=\begin{cases} \delta_{\Rea^d}(x) - \delta_{\Rea^d}(s_n(x)) & \text{if } x \in V_n\setminus W_n, \\
\delta_{\Rea^d}(x)-\sum_{v\in V_{n-1}} \Lambda^d_{2^{-n+1}}(v,x) \delta_{\Rea^d}(v)
& \text{if } x\in W_n\setminus V_{n-1}.
\end{cases}
\]
Let $(x_j)_{j=1}^\infty$ be an arrangement of $\cup_{n=0}^\infty V_n$ such that $(\eta(x_j))_{j=1}^\infty$ is non-decreasing with respect to the lexicographical order. Then $(f(x_j))_{j=1}^\infty$ is a Schauder basis of $\F_p(\Rea^d)$.
\end{Theorem}

\begin{proof}
Given $t\in(0,\infty)$, for $d=1$ define $r_t\colon\Rea\to\Rea$ by
\[
r_t(x)=\min\left\{ 1, \frac{t}{|x|}\right\} x, \quad x\in\Rea.
\]
In general, we define $r_t\colon\Rea^d \to \Rea^d$ by
\[
r_t((x_i)_{i=1}^d)=(r_t(x_i))_{i=1}^d.
\]
The map $r_t$ is $1$-Lipschitz and its range is $B_t:=\{ x \in\Rea ^d \colon \Vert x\Vert_\infty\le t\}$. Let $ S_t\colon \F_p(\Rea^d)\to \F_p(B_t)$ be obtained by linearization of $r_t$. Given $R>0$ such that $t/R\in\Nat$, denote $V_{t,R}=\VV^d_R\cap B_t$, and let $T_R\colon \F_p(\Rea^d)\to \F_p(\VV^d_R) $ and $ T_{t,R}\colon \F_p(B_t)\to \F_p(V_{t,R})$ be the linear maps provided by Theorem~\ref{thm:LipFPB}. Notice that Theorem~\ref{thm:LipFPB} yields a constant $C$ be such that $\Vert T_{t,R}\Vert\le C$ for every $t>0$ and $R$ with $t/R\in\Nat$. Let $L_{\MM,\NN}$ denote the linearization of the inclusion of $\MM$ into $\NN$, and put $L_{R}=L_{\VV^d_R,\Rea^d}$, $L_{t,R}=L_{V_{t,R},B_t}$ and $L'_{t,R}=L_{V_{t,R},\Rea^d}$. We shall prove that
\begin{equation}\label{eq:commutes}
L_{t,R}\circ T_{t,R}\circ S_t = S_t\circ L_{R}\circ T_R, \quad \frac{t}{R}\in\Nat.
\end{equation}
Set $\delta=\delta_{\Rea^d}$, $\delta_t=\delta_{B_t}$, and $\delta_{t,R}=\delta_{V_{t,R}}$.

By Proposition~\ref{prop:FPB}~\ref{prop:FPB:6} a similar argument works for the general case, hence for notational ease we will deal with the case $d=1$. Let $x\in\Rea$. In the case when $|x|\le t$, we have $r_t(x)=x$ and $\Lambda_R(v,x)=0$ unless $|v|\le t$, in which case $r_t(v)=v$. Consequently,
\begin{align*}
S_t(L_{R}(T_R(\delta(x))))
&=\sum_{v\in\VV_R} \Lambda_R(v,x)\delta_t(r_t(v))\\
&=\sum_{v\in\VV_R} \Lambda_R(v,x)\delta_t(v)\\
&=L_{t,R}(T_{t,R}(\delta_t(x)))\\
&=L_{t,R}(T_{t,R}(\delta_t(r_t(x)))).
\end{align*}
In the case when $|x|>t$ we have $\Lambda_R(v,x)=0$ unless $|v|\ge t$ and $\sgn(v)=\sgn(x)$, in which case $r_t(v)=r_t(x)$. Then we have
\begin{align*}
S_t(L_{R}(T_R(\delta(x))))
&=\sum_{v\in\VV_R} \Lambda_R(v,x)\delta_t(r_t(v))\\
&=\sum_{v\in\VV_R} \Lambda_R(v,x)\delta_t(r_t(x))\\
&=\delta_t(r_t(x))\\
&=L_{t,R}(T_{t,R}(\delta_t(r_t(x)))).
\end{align*}
Since $\delta_t(r_t(x))=S_t(\delta(x))$ and $\{ \delta(x) \colon x\in\Rea\}$ spans $\F_p(\Rea)$, \eqref{eq:commutes} holds. Note that the range of $U_{t,R}:= T_{t,R}\circ S_t$ is $\spn\{\delta_{t,R}(x) \colon x\in V_{t,R}\}$. Put $P_{t,R}:=L'_{t,R}\circ T_{t,R}\circ S_t$. By \eqref{eq:commutes}, for every $x\in\Rea^d$ we have
\begin{equation}\label{eq:commutes3}
P_{t,R}(\delta(x)) = \sum_{v\in\VV^d_R} \Lambda^d_R(v,x)\delta(r_t(v)), \quad \frac{t}{R}\in\Nat.
\end{equation}
Thus, for every $v\in \VV^d_R$ we have $P_{t,R}(\delta(v)) = \delta(r_t(v))$, which in combination with \eqref{eq:commutes3} implies that $P_{t,R}$ is a projection with range equal to $\spn\{\delta(x)\colon x\in V_{t,R}\}$. We infer that
\begin{equation}\label{eq:commutes2}
P_{t,R} \circ P_{t',R'} = P_{t',R'} \circ P_{t,R}= P_{t,R}, \quad
\frac{R}{R'}, \frac{t'}{R'}, \frac{t}{R}\in\Nat, \, t'\ge t>0.
\end{equation}
Indeed, $P_{t',R'} \circ P_{t,R}= P_{t,R}$ follows from the fact that the range of $P_{t,R}$ is contained in the range of $P_{t',R'}$ and we have
\begin{align*}
P_{t,R}\circ P_{t',R'} & = L'_{t,R}\circ T_{t,R}\circ S_t\circ L_{B_{t'},\Rea^d}\circ S_{t'}\circ L_{R'}\circ T_{R'}\\
& = L'_{t,R}\circ T_{t,R}\circ S_t\circ L_{R'}\circ T_{R'}\\
&= P_{t,R},
\end{align*}
where in the first equality we used \eqref{eq:commutes}, in the second we used the observation that $S_t\circ L_{B_{t'},\Rea^d}\circ S_{t'} = S_t$ and the third equality follows from Proposition~\ref{prop:FPB}~\ref{prop:FPB:5}.

Hence, if for $n\in\Nat_*$ we put
\[
P_{2n}=P_{k_n,2^{-n}} \text{ and }
P_{2n-1}= P_{k_{n-1},2^{-n}},
\]
we have $P_j\circ P_{j'}=P_{j'}\circ P_j=P_j$ whenever $-1\le j\le j'$. Therefore, there is Schauder decomposition $(Y_j)_{j=-1}^\infty$ of $\F_p(\Rea^d)$ whose associated projections are $(P_j)_{j=-1}^\infty$. Moreover, for all $n\in\Nat_*$ the range of $P_{2n}$ is $\spn\{\delta(x) \colon x\in V_n\}$ and the range of $P_{2n-1}$ is $\spn\{\delta(x) \colon x\in W_n\}$.

Similar arguments as in the proof of Theorem~\ref{thm:basis01d} show that there is a constant $C=C(p,d)$ such that
\[
(f(x))_{x\in W_n\setminus V_{n-1}}
\]
is a $C$-unconditional basis of this $Y_{2n-1}$ for all $n\in\Nat_*$. Let $n\in\Nat_*$. Note that $r_t(x)=r_t(s_n(x))$ for all $t$ with $2^n t\in\Nat$ and all $x\in V_n\setminus W_n$ with $\Vert x \Vert_\infty>t$. In particular,
\[
r_{k_{n-1}}(x)=r_{k_{n-1}}(s_n(x)), \quad x\in V_n\setminus W_n.
\]
Therefore, $P_{2n-1}(f(x))=0$ for every $x\in V_n\setminus W_n$, which in turn implies that $f(x)\in Y_{2n}$. An inductive argument yields that for every $x\in V_n\setminus W_n$,
\[
g(x):=\delta(x)-\delta(r_{k_{n-1}}(x))
\]
is a linear combination of the family (of nonzero vectors)
\[
\BB_n=(f(y))_{y\in V_n\setminus W_n}.
\]
Since, by an argument similar to that used in the proof of Theorem~\ref{thm:basis01d}, $(g(x))_{x\in V_n\setminus W_n}$ generates $Y_{2n}$, $\BB_n$ also generates $Y_{2n}$. Let $F\subset V_n\setminus W_n$ be such that
\[
V_{t-2^{-n},2^{-n}}\setminus W_n
\subset F
\subset V_{t, 2^{-n}} \setminus W_n
\]
for some $t\in(k_{n-1}, k_n]\cap 2^{-n}\Int$. Set $\alpha=(t,2^{-n})$ and let $\widehat{\alpha}$ be the ``predecesor'' of $\alpha$ given by $\widehat{\alpha}=(t-2^{-n},2^{-n})$. Define $r_{\alpha,F}\colon V_{\alpha}\to \Rea^d$ by
\[
r_{\alpha,F}(x)=\begin{cases} r_{t-2^{-n}}(x) &\text{ if } x\in V_{\alpha} \setminus F, \\ x &\text{ if } x\in F. \end{cases}
\]
If $x\in V_{\widehat{\alpha}}$ then $r_{t-2^{-n}}(x)=x$. Given $x\in V_\alpha$ there is $z\in V_{\widehat{\alpha}}$ with $\vert x-z\vert_\infty=2^{-n}$, and $V_{\alpha}$ is $2^{-n}$-separated. Hence, by Lemma~\ref{lem:18}, $r_{\alpha,F}$ is $C_p$-Lipschitz, where $C_p=(1+2^{1/p})$. Let $U_{\alpha,F}\colon \F_p(V_\alpha)\to \F_p(\Rea^d)$ be the linear map defined by $\delta_\alpha(x) \mapsto \delta(x)$ if $x\in F$ and $\delta_\alpha(x) \mapsto \delta (r_{t-2^{-n}}(x))$ if $x\in V_{\alpha}\setminus F$. Set $U_\alpha= T_\alpha\circ S_t$. We infer that, if
\[
Q_{\alpha,F}=U_{\alpha,F}\circ U_\alpha|_{Y_{2n}},
\]
then $\Vert Q_{\alpha,F} \Vert\le C C_p$. Note that $U_\alpha(\delta(x))=\delta_\alpha(x)$ for all $x\in V_\alpha$ and that $U_\alpha(\delta(x))=U_\alpha(\delta(s_n(x)))$ for all $x\in V_n\setminus V_\alpha$. If $x\in V_\alpha$, then $s_n(x)\in V_{\widehat{\alpha}}$, and so $Q_{\alpha,F}(f(x))=f(x)$ for every $f\in F$. If $x\in V_\alpha\setminus V_{\widehat{\alpha}}$, then $s_n(x)=r_{t-2^{-n}}(x)$. Consequently, $Q_{\alpha,F}(f(x))=0$ for every $x\in V_\alpha\setminus F$. Finally, we deduce that $Q_{\alpha,F}(f(x))=0$ for every $x\in V_n\setminus V_{\alpha}$. We infer that, if $(x_j)_{j=1}^{|V_n|-|W_n|}$ is an arrangement of $V_n\setminus W_n$ with $(|x_j|)_{j=1}^{|V_n|-|W_n|}$ non-decreasing, then $(f(x_j))_{j=1}^{|V_n|-|W_n|}$ is a Schauder basis of $Y_{2n-1}$ with basis constant at most $CC_p$.
\end{proof}

\section{Open problems}
\noindent
By \cite{K86}*{Theorem 5.2}, there exists a subspace $Z$ of $\ell_p$ whose Banach envelope is isomorphic to $L_1$, so we would like to know whether $\F_p(\Rea)$ is different from the subspace $Z$, whose existence is guaranteed by a general abstract construction.

\begin{Question}
Let $p\in(0,1)$. Is $\F_p(\Rea)$ isomorphic to a subset of $\ell_p$?
\end{Question}

Once we have told apart $\F_p(\Nat)$ from $\ell_p$ for $p<1$, it is natural to go on with the topic \ref{QB} by trying to determine how many non-isomorphic Lipchitz free $p$-spaces one can obtain from subsets of $\Nat$. Note that if $\NN$ is a subset of $\Nat$ then $\F_p(\NN)$ is a complemented subspace of $\F_p(\Nat)$ by \ref{three:complemeted}. So, this problem connects with that of characterizing the complemented subspaces of $\F_p(\Nat)$. Let us illustrate the question with an example. Suppose that $\NN\subset\Nat$ contains arbitrarily long chains of consecutive integers.
Then, there is $(a_k)_{k=0}^\infty$ in $\Nat$ such that $a_k+\Nat_{2^k-1}\subset \NN$ and $2(a_{k}+2^{k}-1) \le a_{k+1}$ for all $k\in\Nat_*$. Set
\[
\NN_0= \bigcup_{k=0}^\infty a_k+\Nat_{2^k-1}.
\]
Combining \cite{AACD2019}*{Lemma 2.1} with \cite{AACD2020}*{Theorem 5.8} yields $\F_p(\NN_0)\simeq \F_p(\Nat)$. Then, by \ref{three:complemeted}, $\F_p(\NN)$ is complemented in $\F_p(\Nat)$ and, the other way around, $\F_p(\Nat)$ is complemented in $\F_p(\NN)$. Taking into account \ref{three:isomorphismsNd}, Pe{\l}czy\'nski's decomposition method yields $\F_p(\NN)\simeq\F_p(\Nat)$.

\begin{Question}Let $0<p<1$. Does there exist $\NN\subset\Nat$ such that $\F_p(\NN)$ is neither isomorphic to $\ell_p$ nor to $\F_p(\Nat)$?
\end{Question}

A well-known problem in Geometric Functional Analysis is whether $\F(\Nat^2)$ is isomorphic to $\F(\Nat^3)$ or, more generally, whether $\F(\Nat^d)$  is isomorphic to $\F(\Nat^{d+1})$ for $d\ge 2$.
As, by \cite{AACD2018}*{Proposition 4.20}, if  $\F_p(\Nat^d)\simeq\F_p(\Nat^{d+1})$ for some $p<1$ then $\F(\Nat^d)\simeq\F(\Nat^{d+1})$, investigating in more depth the geometry of $\F_p(\Nat^d)$ for $p<1$ and $d\ge 2$ could shed some light into this important question. Similarly,  telling apart $\F_p(\Rea^d)$ from $\F_p(\Rea^{d+1})$ for $p<1$ could be easier than telling apart $\F(\Rea^d)$ from $\F(\Rea^{d+1})$. 
%Since $\F_p(\Rea^d)\simeq \F_p(S_{\Rea^{d+1}})$  
By \cite{AACD2020}*{Theorem 4.21} the latter  is equivalent to proving that Lipschitz $p$-free spaces over Euclidean spaces are not isomorphic to Lipschitz $p$-free spaces over its spheres.

\begin{Question}Let $0<p\le 1$ and $d\ge 2$. Is $\F_p(\Nat^d)$ is isomorphic to $\F_p(\Nat^{d+1})$?
\end{Question}

\begin{Question}Let $0<p\le 1$ and $d\ge 3$. Is $\F_p(\Rea^d)$ is isomorphic to  $\F_p(S_{\Rea^d})$?
\end{Question}

%\bibliography{AnsoBiblio}{}
%\bibliographystyle{plain}

% \bib, bibdiv, biblist are defined by the amsrefs package.

\end{document}